\documentclass[12pt,a4paper]{amsart}
\usepackage{amsfonts,amsthm,amssymb,mathrsfs,amscd,amsmath,url}
\usepackage[utf8]{inputenc}
\usepackage[T1]{fontenc}
\usepackage{dsfont}

\usepackage{pdfcomment}

\usepackage[main=english]{babel}

\usepackage{colortbl}

\usepackage[a4paper]{geometry}
\usepackage{natbib}
\usepackage{graphicx,caption}
\usepackage{comment}

\usepackage{epigraph}

\usepackage{mathtools}
\mathtoolsset{showonlyrefs}

\usepackage{enumitem}
\usepackage{relsize}
\usepackage{varioref}
\usepackage{hyperref}
\usepackage{cleveref}

\newcommand{\N}{\mathbb{N}}
\newcommand{\Z}{\mathbb{Z}}
\renewcommand{\P}{\mathbb{P}}
\newcommand{\R}{\mathbb{R}}

\newcommand{\C}{\mathbb{C}}

\newtheorem{theorem}{Theorem}
\newtheorem*{theorem*}{Theorem}

\newtheorem*{conj}{Conjecture}

\newtheorem*{definition}{Definition}

\newtheorem*{remark*}{Remark}

\newtheorem{lemma}{Lemma}
\newtheorem*{lemma*}{Lemma}

\newtheorem{prop*}{Proposition}

\newtheorem*{corollary*}{Corollary}
\newtheorem{corollary}{Corollary}

\newcommand{\congruence}[3]{\ensuremath{{#1}\equiv {#2}\bmod{#3}}}

\newcommand{\1}[1]{\mathds{1}\left[#1\right]}

\newcommand{\mycomment}[1]{}
\renewcommand{\epsilon}{\varepsilon}
\renewcommand{\phi}{\varphi}
\newcommand{\Chi}{\mathcal{X}}

\usepackage{xcolor}
\usepackage{comment}
\includecomment{comment}

\title{Large central values of Dirichlet L-functions in cosets}
\author{Ivan Ermoshin}
\date{}

\begin{document}

\maketitle
\begin{abstract}
    In this paper we consider the distribution of large central values of Dirichlet $L$-functions over cosets of the group of characters modulo $q$ via Soundararajan's resonator method.
\end{abstract}

\tableofcontents

\section{Introduction}

\vspace{1em}
Let $\chi\in\mathcal{X}$, where $\mathcal{X}$ denotes a group of Dirichlet characters modulo $q>1$. Then the $L$-function attached to $\chi $ is defined for $\Re s>1$ by $$L(\chi,s)=\sum_{n=1}^\infty \frac{\chi(n)}{n^s}$$
and can be analytically continued to an entire function satisfying a certain functional equation except the principal one $\chi_0$, in which case it is meromorphic
with a pole at $s=1$. See subsection \ref{fe} for the functional equation for primitive characters. 

It is known that the zeros of such $L$-functions either are negative integers or lie in the critical strip $0<\Re s<1$. The generalized Riemann hypothesis states that all the non-trivial zeros lie on the critical line $\Re s=\frac{1}{2}.$ On top of that, we are interested in statistical behaviour of $L$-functions inside the critical strip. 

We are interested not only in the behaviour of an individual $L$-function, but also in behaviour on average over some families. In particular we consider central values of family of Dirichlet $L$-functions of  modulo $q$
$$\{L(\chi,1/2)\mid \chi\in\mathcal{X}\}.$$

There are plenty of results and conjectures on statistical behavior of $L$-functions. 

For example, the well known Selberg central limit theorem for Riemann zeta \cite{clt}, that says that the limiting distribution of normalized zeta functions on critical line is the normal distribution.
        $$\lim_{T\to\infty}\text{meas}\{T\le t\le 2T\mid \log|\zeta(1/2+it)|\ge V\sqrt{\frac{1}{2}\log\log T}\}=\frac{1}{2\pi}\int_V^\infty e^{-u^2/2}du.$$

    Later Selberg proved that the same statement holds for $L$-function attached to primitive Dirichlet characters, see \cite{zbMATH00421677}.

    Due to results on moments of $L$-functions it is believed that $L$-functions are not too large on the critical line. More precisely, there is the following conjectured bound due to Lindel\"{o}f, which follows from GRH.
    $$|L(\chi, 1/2+it)|\ll (q|t|+1)^\epsilon.$$

\subsection{Large values in cosets of characters}
In this paper we study extreme central values of Dirichlet $L$-functions and their distributions inside the characters group modulo $q$. 

Mostly we follow Soundararajan's paper \cite{sound}, where he proves certain bounds for extreme values for zeta on the critical line and for central values over families of quadratic characters and $L$-functions attached to cusp forms. 

It is important to say that there are some heuristics based on statistical fact that distances between zeros behaves like distances between eigenvalues of random matrices, it is called \emph{type of a family} (either unitary, orthogonal or symplectic), unitary in case of all characters modulo $q$. It allows to construct precise enough probabilistic model which involves both random matrices and approach with random Euler product which produces widely believed conjectures.

This leads to the following conjecture made by Gonek, Hughes and Farmer in \cite{farmer2006maximum}.

\begin{conj}[Gonek-Hughes-Farmer]
    $$\max\limits_{\chi\bmod q}|L(\chi,1/2)|=\exp((\sqrt{1/2}+o(1))\sqrt{\log q\log\log q}).$$
\end{conj}

While the best known order of magnitude is due to Bret\`eche and Tenenbaum \cite{bt}.
\begin{theorem*}[Bret\`eche, Tenenbaum]
$$\max\limits_{\substack{\chi\bmod q\\\chi\text{even}}}|L(\chi,1/2)|\gg\exp\left((1+o(1))\sqrt{\frac{\log q\log\log\log q}{\log\log q}}\right).$$
\end{theorem*}

We will be interested in large values when $\chi$ is restricted to a coset of a subgroup of $(\Z/q\Z)^\times$. We propose the following conjecture:
\begin{conj}
    Any coset of a subgroup $H\le\mathcal{X}$ such that $\#H\gg q^\epsilon$ contains an extreme value over the whole group, i.e. for any $\chi_1\in\Chi/H$
    $$\max\limits_{\chi\in \chi_1H}|L(\chi,1/2)|=\exp((\sqrt{1/2}+o(1))\sqrt{\log q\log\log q}).$$
    
\end{conj}

The conjecture resembles a straightforward corollary of Bourgain-Konyagin's theorem (see \cite{BOURGAIN200375} theorem 2.1) in the sense that we expect some random behaviour to happen in small subgroups of a large group.
\begin{corollary*}[Bourgain-Konyagin]
    For any coset $hH$ of a subgroup $H\le (\Z/q\Z)^\times$ such that $\#H\gg q^\epsilon$
    $$\text{set of rationals }\left\{\frac{a}{q}\mid a\in hH\right\}\text{ is equidistributed in common sense on $[0,1]$ interval.}$$
\end{corollary*}

The behaviour of central values of $L$-functions associated to subgroups have been also studied by Munsch and Shparlinski \cite{munschshpa}. Using mollifiers they have proved non-vanishing for a positive proportion of character in a subgroup of sub-logarithmic index for prime conductors. Not to mention restricting to almost all primes allows them to state the positive proportion result for much smaller subgroups of polynomial index.

\subsection{Overview of results}

By the original resonator method with some adjustments we can obtain large values in subgroups of infinite index, in particular for any power of $\log q$.

\begin{theorem}\label{thm1}
    For large prime $q$ and $H\le\mathcal{X}$ satisfying
    $\log[\mathcal{X}:H]=o\left(\mathcal{L}\right)$ we have
    $$\max\limits_{\chi\in H} |L(\chi,1/2)|\ge \exp\left(\left(\frac{1}{\sqrt{3}}+o(1)\right)\mathcal{L}\right),\text{ where $\mathcal{L}=\sqrt{\log q/\log\log q}.$}$$
\end{theorem}

The method however breaks down for non-trivial coset. On the other hand when the index is finite, by using a much shorter resonator, we obtain values in an arbitrary coset nearly as large as in the subgroup setting.

\begin{theorem}\label{thm2}
    For large prime $q$ let $H\le\mathcal{X}$ such that $[\mathcal{X}:H]=o\left(\frac{\log q}{\log\log q}\right)$.

    Then if a coset $\chi_1H$ contains both even and odd characters we have that
    $$\max_{\chi\in\chi_1H}|L(\chi,1/2)|\ge\exp{\left(\frac{1}{\sqrt{2[\mathcal{X}:H]-1}}\left(\frac{1}{\sqrt{3}}+o(1)\right)\mathcal{L}\right)}.$$

    While if a coset $\chi_2H$ is contained in even or odd characters the bound becomes
    $$\max_{\chi\in\chi_2H}|L(\chi,1/2)|\ge\exp{\left(\frac{1}{\sqrt{[\mathcal{X}:H]-1}}\left(\frac{1}{\sqrt{3}}+o(1)\right)\mathcal{L}\right)}.$$
\end{theorem}

But we do not get good bounds for larger indices, so the "usable" range is not as wide as one can ask for. It would be interesting if one can remove this restriction.

A noticeable application is even and odd character subgroups, in this case we get
\begin{corollary*}
For large prime $q$ there exists a character $\chi$ modulo $q$ such that
\begin{equation}
        |L(\chi,1/2)|\ge\exp\left(\left(\frac{1}{\sqrt{3}}+o(1)\right)\mathcal{L}\right).
    \end{equation}
    Moreover, one can pick such a character either even or odd.
\end{corollary*}

The theorem gives a weaker bound than Bret\`eche and Tenebaum's one. They apply the resonator method to $|L(\chi,1/2)|^2$ while we do it for $L(\chi,1/2)$. Therefore they get so called G\'al sums to optimize, first considered by Bondarenko and Seip in \cite{bondarenkoseip} while working on large values of zeta. Due to their approach Bret\`eche and Tenebaum use positivity trick more than we do, so the method does not work for odd characters which is the simplest example of a non-trivial coset.

Also we would like to mention the paper of Munsch and Shparlinski \cite{munschshpa}. Applying their results on small solutions of congruences does not improve our general bounds, though we are able to strengthen the result for almost all primes.

\begin{remark*}
    We consider only prime $q$ for simplicity.
For arbitrary $q$ one firstly has to deal with non-primitive characters which have different approximate functional equation due to a different conductor, but that is not a major issue; secondly, the condition $(n,q)=1$ must be added to the support of resonator function in Theorem \ref{sound}, and this condition may modify the result and harden the proof. Nevertheless, it should be possible to prove at least the full group case, probably also even and odd cosets.
\end{remark*}

\section{Tools}

{\bf Notations:} The letters $p,q,r$ stand for prime numbers and $\P$ be the set of all primes. All logarithms are to base $e$.
The signs $\ll$, $\gg$ are usual Vinogradov signs, $f \asymp g$ means that $f \ll g$ and $g \ll f$, while $f \sim g$ means that $f \ll g$ and $g \ll f$ and the constants in both asymptotic signs coincide. The functions $(a,b)$ and $[a,b]$ stand for the greatest common divisor and the least common multiple of $a$ and $b$ respectively.

$e(x)=\exp(2\pi ix).$

Let $\mathcal{L}=\frac{\sqrt{\log q}}{\sqrt{\log\log q}}.$

Let $\mathcal{X}$ be the group of Dirichlet characters modulo $q$.

Then denote cosets of even and odd characters as $\mathcal{X}_0$ and $\mathcal{X}_1$ respectively.

\subsection{Auxiliary lemmas}
Here we explain some well known results not related to number theory that will be used in the following sections.
\begin{lemma}[Rearrangement inequality]\label{rearrangement}
For two arbitrary finite sequences of real numbers $x_1\le\ldots\le x_n$ and $y_1\le\ldots\le y_n$ and any permutation $\sigma\in S_n$ we have
$$x_1y_n+\cdots+x_ny_1\le\sum_{k=1}^n x_ky_{\sigma(k)}\le x_1y_1+\cdots+x_ny_n.$$
\end{lemma}

\subsection{Orthogonality relations}
    There are two types of (common-sense) orthogonalities for Dirichlet characters, for sums over all characters and for sum over all residues. However, we only need sums over characters.

The mentioned orthogonality over values of the argument is stated below.

\begin{lemma*}
    For a natural number $q$ and any character $\chi\in\mathcal{X}$ the following holds
    $$\sum_{a\text{ mod }q}\chi(a)= 
    \begin{cases}
        \phi(q),& \chi=\chi_0,\\
        0,&\text{otherwise.}
    
    \end{cases}$$
\end{lemma*}

Let us denote by $\ker H$ a set of residues $h$ where $\forall\chi\in H:\;\chi(h)=1$. That is,
\begin{definition}
$\ker H=\{h\in (\Z/q\Z)^\times\mid \forall\chi\in H:\; \chi(h)=1\}$.
\end{definition}

It is easy to see that $\#\ker H\cdot\#H=\phi(q)$. We will prove it later.

Algebraically speaking, any finite abelian group is isomorphic to its group of characters, and the proposition illustrates a natural one-to-one correspondence between subgroups of the group, and subgroups of the group of characters.

Note that Dirichlet characters are characters of $(\Z/q\Z)^\times$ extended to $\Z$.

\begin{lemma}\label{principal_orthogonality}
$$\sum_{\chi\in \chi_1H}\chi(a)=\begin{cases}
                        \chi_1(a)\#H,&a\in\ker H,\\
                        0,& \text{otherwise.}
                    \end{cases}$$
\end{lemma}
\begin{proof}
    The first case is obvious.

    We obtain the second case as usual: pick a character $\psi\in H$ such that $\psi(a)\ne1$, then
    $$\psi(a)\sum_{\chi\in H}\chi(a)=\sum_{\chi\in H}(\psi\chi)(a)= \sum_{\chi\in H}\chi(a),$$
    $$(\psi(a)-1)\sum_{\chi\in H}\chi(a)=0.$$
\end{proof}

\begin{corollary*}
    For any natural $q$ and any subgroup $H\le\mathcal{X}$ the following holds
    $$\#\ker H\cdot\#H=\phi(q).$$
\end{corollary*}
\begin{proof}
    Use the orthogonality over values and the generalized version of orthogonality over characters we prove below in different order
    \begin{equation}\begin{split}\#\ker H\#H=\sum_{a\text{ mod }q}\#H\1{a\in\ker H}=\sum_{a\text{ mod }q}\sum_{\chi\in H}\chi(a)=\\=\sum_{\chi\in H}\sum_{a\text{ mod }q}\chi(a)=\sum_{\chi\in H}\phi(q)\1{\chi=\chi_0}=\phi(q).\end{split}\end{equation}
\end{proof}

Define Gauss sum as usual.
\begin{definition}
$\tau(\chi)=\sum\limits_{a\bmod q}\chi(a)e(a/q).$
\end{definition}

To treat the dual sums, we also need the following statement which is an easy corollary of Lemma \ref{principal_orthogonality}.

\begin{lemma}\label{dual_orthogonality}
    $$\sum_{\chi\in\chi_1 H}\tau(\chi)\chi(a)=\chi_1(a)\#H\sum_{h\in\ker H}\chi_1\left(h\overline{a}\right)e\left(\frac{h\overline{a}}{q}\right).$$  
\end{lemma}
\begin{proof}
$$\sum_{\chi\in\chi_1 H}\tau(\chi)\chi(a)=\sum_{\chi\in H}\sum_{n=1}^N \chi_1(a)\chi(a)\chi_1(n)\chi(n)e(n/q)=$$
$$\chi_1(a)\sum_{n=1}^N \chi_1(n) e(n/q)\sum_{\chi\in H}\chi(a)\chi(n)=\chi_1(a)\#H\sum_{n=1}^N \chi_1(n)e(n/q)\1{an\in\ker H}=$$
$$\chi_1(a)\#H\sum_{h\in\ker H}\sum_{\substack{n=1\\\congruence{an}{h}{q}}}^q \chi_1(n)e(n/q).$$
\end{proof}

\subsection{Approximate functional equation}

Let $\chi$ be a primitive character modulo $q$, let $\kappa$ denote the parity of the character, $$\kappa=
\begin{cases}
0,&\chi(-1)=1 \\
1,& \chi(-1)=-1
\end{cases},$$ and $\Lambda$ be a completed $L$-function
$$\Lambda(\chi,s)=\left(\frac{q}{\pi}\right)^{s/2}\Gamma\left(\frac{s+\kappa}{2}\right)L(\chi,s).$$
Then we have the following property (see \cite{iwaniec2004analytic} theorems 4.15 and 5.3 for proof of the two following lemmas respectively).
\begin{lemma*}[Functional equation]\label{fe}
    $$\Lambda(\chi,s)=\frac{\tau(\chi)}{i^\kappa\sqrt{q}}\Lambda(\overline{\chi},1-s),$$
\end{lemma*}

where $\tau(\chi)=\sum\limits_{k\mod q}\chi(k)e(k/q)$ is usual Gauss sum.

When combined with the approximate functional equation, this provides continuation of the $L$-function to the whole complex plane.

\begin{lemma*}[Approximate functional equation]
    For $0<\Re s<1$ and $X>0$ we have
    $$L(\chi,s)=\sum_{n}\frac{\chi(n)}{n^s}V_{\kappa,s}\left(\frac{n\sqrt{\pi}}{X\sqrt{q}}\right)+\frac{\tau(\chi)}{i^\kappa\sqrt{q}}\sum_{n}\frac{\overline{\chi(n)}}{n^{1-s}}V_{\kappa,1-s}\left(\frac{n\sqrt{\pi}X}{\sqrt{q}}\right),$$
    where $V_{\kappa,s}$ is a smooth function defined by
    $$V_{\kappa,s}(y)=\frac{1}{2\pi i}\int\limits_{(\sigma)} y^{-u}\frac{\Gamma\left(\frac{u+s+\kappa}{2}\right)}{\Gamma\left(\frac{s+\kappa}{2}\right)}\frac{du}{u}.$$
\end{lemma*}

Let $V_\kappa(y)=V_{\kappa,1/2}(y)$. For our purposes the following is sufficient.
\begin{corollary} \label{afe}For $X>0$
    $$L(\chi,1/2)=\sum_{n}\frac{\chi(n)}{\sqrt{n}}V_{\kappa}\left(\frac{n\sqrt{\pi}}{X\sqrt{q}}\right)+\frac{\tau(\chi)}{i^\kappa\sqrt{q}}\sum_{n}\frac{\overline{\chi(n)}}{\sqrt{n}}V_{\kappa}\left(\frac{n\sqrt{\pi}X}{\sqrt{q}}\right). $$
\end{corollary}
For further purposes we introduce a notation $L^*(\chi,1/2)$ of this formula for all characters $\chi\text{ modulo } q$, not only primitive ones:
$$L^*(\chi,1/2)=\sum_{n}\frac{\chi(n)}{\sqrt{n}}V_{\kappa}\left(\frac{n\sqrt{\pi}}{X\sqrt{q}}\right)+\frac{\tau(\chi)}{i^\kappa\sqrt{q}}\sum_{n}\frac{\overline{\chi(n)}}{\sqrt{n}}V_{\kappa}\left(\frac{n\sqrt{\pi}X}{\sqrt{q}}\right). $$
We will also use the following bounds for $V_\kappa(s)$
\begin{lemma}\label{v_lemma}
For $y>0$ we have $V_\kappa(y)\ge0,$

    For $0<y<1$ we have $V_\kappa(y)=1+O(y^{1/2}),$
    
    For $y>1$ we have $V_\kappa(y)\ll\exp(-y).$
\end{lemma}
\begin{proof}
    We start with $$e^{-t^2}=\frac{1}{2\pi i}\int\limits_{(\sigma)}\Gamma(u)\frac{du}{t^{2u}},$$
    by change of variables we get
    $$2t^{ \kappa-1/2}e^{-t^2}=\frac{1}{2\pi i}\int\limits_{(\sigma)}\Gamma\left(\frac{u+1/2+\kappa}{2}\right)\frac{du}{t^{u+1}},$$
    integrate with respect to $t$ from $y$ to $\infty$ and divide by constant $\Gamma$
    $$\frac{2}{\Gamma\left(\frac{1/2+\kappa}{2}\right)}\int_y^\infty t^{\kappa-1/2}e^{-t^2}dt=\frac{1}{2\pi i}\int\limits_{(\sigma)}y^{-u}\frac{\Gamma\left(\frac{u+1/2+\kappa}{2}\right)}{\Gamma\left(\frac{1/2+\kappa}{2}\right)}\frac{du}{u}=V_\kappa(y).$$

    From here it is clear that $V_\kappa$ is positive.

    Also it is clear that $V_\kappa(0)=1,$
    hence for the first bound we consider
    $$V_\kappa(0)-V_\kappa(y)=\frac{2}{\Gamma\left(\frac{1/2+\kappa}{2}\right)}\int_0^y t^{\kappa-1/2}e^{-t^2}dt\ll\int_0^y t^{2(\kappa/2-3/4)}e^{-t^2}d(t^2)\ll$$
    $$\int_0^{y^2}t^{\kappa/2-3/4}e^{-t}dt\ll\int_0^{y^2}t^{\kappa/2-3/4}dt=\frac{y^{1/2+\kappa}}{\kappa/2+1/4}\ll y^{1/2}.$$

    For the upper bound for $V_\kappa(y)$ we use trivial bounds and $y>1$
    $$V_\kappa(y)\ll\int_0^\infty (t+y)^{\kappa-1/2}e^{-(t+y)^2}dt\ll ye^{-y^2}\int_0^\infty (t/y+1)e^{-t^2}dt\ll ye^{-y^2}\ll e^{-y}.$$
\end{proof}

\subsection{Soundararajan's optimisation theorem}\label{sound}

In his paper $\cite{sound}$ Soundararajan proves the following optimisation theorem, which he uses then to prove some estimates for values of zeta on the critical line. So we do find it convenient that we use the same theorem for character groups, since both of these cases may be considered as unitary families.

\begin{theorem*}[Soundararajan]
    For large $N$ we have
    \begin{equation}
        \max\limits_{r:\N\to\C}\sum\limits_{mk\le N}\frac{r(m)\overline{r(mk)}}{\sqrt{m}}\mathlarger{/}\sum\limits_{n\le N}|r(n)|^2=\exp{\left(\frac{\sqrt{\log N}}{\sqrt{\log\log N}}+O\left(\frac{\sqrt{\log N}}{\log\log N}\right)\right)}.
    \end{equation}
\end{theorem*}

\section{Trivial coset}

Here we study extreme values over a trivial coset and we will prove Theorem \ref{thm1} in the introduction. We use some positivity argument as well as Bret\`eche and Tenenbaum in \cite{bt}.

The main difficulty in this and the next section is the fact that the congruence condition arising from orthogonality relations contains an arbitrary element from the kernel of a subgroup, whose size is seemingly impossible to control without any predefined properties of the subgroup.

Following the method in Soundararajan's paper \cite{sound} we define
$$M_1=\sum_{\chi\in H}|R^2(\chi)|,$$
$$M_2=\sum_{\chi\in H}L(\chi,1/2)|R^2(\chi)|,$$
where $R(\chi)=\sum\limits_{n\le N}r(n)\chi(n)$.

We expect to lose some order of magnitude bounding $M_1$ as we do below, though it results in a better final bound than the alternative way, which is presented in the non-trivial coset section.
$$M_1=\sum\limits_{n_1,n_2\le N}r(n_1)\overline{r(n_2)}\sum_{\chi\in H} \chi(n_1)\overline{\chi(n_2)}=\#H\sum_{h\in \ker H}\sum\limits_{n_1,n_2\le N}r(n_1)\overline{r(n_2)}\1{\congruence{hn_1}{n_2}{q}}\le$$
$$\#H\sum_{h\in \ker H}\sum\limits_{n_1,n_2\le N}|r(n_1)||r(n_2)|\1{\congruence{hn_1}{n_2}{q}}\le\phi(q)\sum_n r(n)^2,$$
where in the last step we applied rearrangement inequality (Lemma \ref{rearrangement}).

With $M_2$ we have to do some adjustments since the approximate functional equation holds only for primitive characters, so
$$M_2=(L(\chi_0,1/2)-L^*(\chi_0,1/2))|R^2(\chi_0)|+\sum_{\chi\in H}L^*(\chi,1/2)|R^2(\chi)|,$$
where $L^*(\chi,1/2)$ is the value given by the approximate functional equation (Corollary \ref{afe}).

Note that $|R^2(\chi_0)|\ll N\sum |r(n)|^2$ by Cauchy-Schwarz, $L(\chi_0,1/2)$ is bounded by $2\zeta(1/2)$ and $L^*(\chi_0,1/2)$ can be bounded trivially
\begin{equation}\label{chi_0}
    \begin{split}
L^*(\chi_0,1/2)=\sum_n \frac{\chi_0(n)}{\sqrt{n}}\left(V_0\left(\frac{n\sqrt{\pi}}{X\sqrt{q}}\right)+V_0\left(\frac{n\sqrt{\pi}X}{\sqrt{q}}\right)\right)\ll\\\sum_{n\le X\sqrt{q}/\sqrt{\pi}}\frac{1}{\sqrt{n}}+\sum_{n> \sqrt{q}/(\sqrt{\pi}X)}\frac{1}{\sqrt{n}}\exp{\left(-\frac{n\sqrt{\pi}X}{\sqrt{q}}\right)}\ll X^{1/2}q^{1/4}.
    \end{split}
\end{equation}

Therefore
$$M_2+O\left(NX^{1/2}q^{1/4}\sum |r(n)|^2\right)=\sum_{n_1,n_2\le N}r(n_1)\overline{r(n_2)}\sum_{\chi\in H} \chi(n_1)\overline{\chi(n_2)}L^*(\chi,1/2).$$
Consider the internal sum, there are two possible options: either $H$ consists of even characters, or the subset of even characters of $H$ is a subgroup with $\ker(H\cap\mathcal{X}_0)=\ker H\cup-\ker H$ and the subset of odd characters is its coset.

In the first case we get by Lemma \ref{principal_orthogonality} and Lemma \ref{dual_orthogonality}
$$\sum_{\chi\in H}\chi(n_1)\overline{\chi(n_2)}\left(\sum_n \frac{\chi(n)}{\sqrt{n}}V_0\left(\frac{n\sqrt{\pi}}{X\sqrt{q}}\right)+\frac{\tau(\chi)}{\sqrt{q}}\sum_n \frac{\overline{\chi(n)}}{\sqrt{n}}V_0\left(\frac{n\sqrt{\pi}X}{\sqrt{q}}\right)\right)=$$

$$\#H\sum_n \frac{\sum\limits_{h\in\ker H}\1{\congruence{h nn_1}{n_2}{q}}}{\sqrt{n}}V_0\left(\frac{n\sqrt{\pi}}{X\sqrt{q}}\right)+\frac{\#H}{\sqrt{q}}\sum_n \frac{\sum\limits_{h\in\ker H}e\left(\frac{hn_1\overline{nn_2}}{q}\right)}{\sqrt{n}}V_0\left(\frac{n\sqrt{\pi}X}{\sqrt{q}}\right).$$

While in the second it is morally the same but has some additional contribution
$$\frac{\#H}{2}\sum_n \frac{\sum\limits_{h\in\ker H\cup-\ker H}\1{\congruence{h nn_1}{n_2}{q}}}{\sqrt{n}}(V_0+V_1)\left(\frac{n\sqrt{\pi}}{X\sqrt{q}}\right)+
$$$$\frac{\#H}{2\sqrt{q}}\sum_n \frac{\sum\limits_{h\in\ker H\cup\{-\ker H\}}e\left(\frac{hn_1\overline{nn_2}}{q}\right)V_0\left(\frac{n\sqrt{\pi}X}{\sqrt{q}}\right)+\chi_1\left(\frac{hn_1\overline{nn_2}}{q}\right)e\left(\frac{hn_1\overline{nn_2}}{q}\right)V_1\left(\frac{n\sqrt{\pi}X}{\sqrt{q}}\right)}{\sqrt{n}},$$
where $\chi_1$ is an odd character from $H.$

Now, we bound the error term: it comes from the dual sums and the $V_i$-tails in the principal ones.

Firstly, we bound $\chi_1$ and exponents trivially by $1$. Also note that $\#H\cdot\#\ker H=\phi(q)$.

 Consider $1<y$. In the terms coming from dual sum it means that $\frac{n\sqrt{\pi}X}{\sqrt{q}}>1$, so its contribution is bounded by 
\begin{equation}\label{errorterm}\begin{aligned}\frac{\phi(q)}{\sqrt{q}}\sum_{n_1,n_2\le N}r(n_1)\overline{r(n_2)}\sum_{n>\frac{\sqrt{q}}{\sqrt{\pi}X}}\frac{1}{\sqrt{n}}\exp\left(-\frac{n\sqrt{\pi}X}{\sqrt{q}}\right)\ll\\\frac{\phi(q)}{\sqrt{q}}\sum_{n_1,n_2\le N}r(n_1)\overline{r(n_2)}\int_{\frac{\sqrt{q}}{\sqrt{\pi}X}}^\infty\frac{1}{\sqrt{t}}\exp\left(-\frac{t\sqrt{\pi}X}{\sqrt{q}}\right)dt\ll\\\frac{\phi(q)}{q^{1/4}X^{1/2}}\sum_{n_1,n_2\le N}r(n_1)\overline {r(n_2)}\ll\phi(q)\sum_n |r(n)|^2\frac{N}{q^{1/4}X^{1/2}}.\end{aligned}\end{equation}
Note that here we could not afford the same trick as below since this error term will be the dominating one and extra $q^\epsilon$ would be crucial.

In the principal terms with $1<y$ we split the sum with $\frac{Xq^{1/2+\epsilon}}{\sqrt{\pi}}$, so the tail of the series contributes $q^{-1000}$, while the finite part is estimated trivially assuming $\exp(\ldots)=O(1)$.
$$\phi(q)\sum_{n_1,n_2\le N}|r(n_1){r(n_2)}|\sum_{\substack{n>\frac{X\sqrt{q}}{\sqrt{\pi}}\\\congruence{\pm n_1n}{ n_2}{q}}}\frac{1}{\sqrt{n}}\exp{\left(-\frac{n\sqrt{\pi}}{X\sqrt{q}}\right)}\ll \phi(q)\sum\limits_{n}|r(n)|^2\cdot\frac{X^{1/2}N}{q^{3/4-\epsilon}}.$$

Now consider the remaining case $y<1$, here we have $V_i(y)=1+O(y^{1/2})$.

The dual sums give the second dominating error term, the same as the previous one.
$$\frac{\phi(q)}{\sqrt{q}}\sum_{n_1,n_2\le N}|r(n_1)\overline{r(n_2)}|\sum_{n<\frac{\sqrt{q}}{\sqrt{\pi}X}}\frac{1}{\sqrt{n}}\left(1+\left(\frac{n\sqrt{\pi}X}{\sqrt{q}}\right)^{1/2}\right)\ll \frac{\phi(q)}{q^{1/4}X^{1/2}}N\sum\limits_{n\le N}|r(n)|^2.$$

Regarding the principal term we note that the contribution of all the terms is positive by Lemma \ref{v_lemma}, so we can omit everything except $h=1$ since we only care about the real part.
$$\sum\limits_{h\in\ker H\cup-\ker H}\1{\congruence{h nn_1}{n_2}{q}}\ge\1{\congruence{nn_1}{n_2}{q}}.$$

From here on we observe that the congruence $\congruence{-nn_1}{n_2}{q}$ has no solutions and $\congruence{nn_1}{n_2}{q}$ implies equality $nn_1=n_2$ as long as $X\sqrt{q/\pi}N<q$, i.e. $XN<\sqrt{\pi}q^{1/2}.$

Therefore the error term coming from principle sums with $y<1$ is
$$\phi(q)\sum_{nn_1\le N}\frac{r(n_1)r(nn_1)}{\sqrt{n}}\left(\frac{n}{X\sqrt{q}}\right)^{1/2}\ll \frac{\phi(q)}{(X\sqrt{q})^{1/2}}\sum_{\substack{k\le N\\ d\mid k}}r(d)r(k)\ll$$
by the rearrangement inequality and divisor function bound
$$\frac{\phi(q)}{X^{1/2}{q}^{1/4-\epsilon}} \sum_{n\le N}|r(n)|^2.$$

Then the resulting error term is
\begin{equation}\label{et}
\#H\sum_{n\le N}|r(n)|^2 \left(\frac{N}{q^{1/4}X^{1/2}}+\frac{X^{1/2}N}{q^{3/4-\epsilon}}+\frac{N}{q^{1/4}X^{1/2}}+\frac{1}{X^{1/2}q^{1/4-\epsilon}}+\frac{NX^{1/2}q^{1/4}}{\phi(q)}\right),
\end{equation}
where the first four terms comes as in the text: the first one is dual with $y>1$, the last one is principal with $y<1$, and the fifth term is from the issue with $L(\chi_0)$ having a different functional equation.

Although the main term coming from the 1 in $V_i(y)=1+O(y^{1/2})$ is just
$$\text{MT}=\#H\sum_{nn_1\le N}\frac{r(n)r(nn_1)}{\sqrt{n}}.$$

In order to obtain the best possible bound, we want to make $N$ as large as we could.
The expected order of $M_2/M_1$ is about $\sqrt{\log q}$, so the error term should be $O(1)$, since allowing anything larger would not improve the result. Recall we also have the condition $XN<\sqrt{\pi}q^{1/2}$, to simplify it one may assume $XN=q^{1/2}$. Now consider the first term $\frac{N}{q^{1/4}X^{1/2}}=\frac{N^{3/2}}{q^{1/2}}$, so the largest value of $N$ is $N=q^{1/3}$, and $X=q^{1/6}$, hence
$$M_2-\text{MT}\sim\#H\sum_{n\le N}|r(n)|^2\ll \phi(q)\sum_{n\le N}|r(n)|^2=M_1.$$

So by the theorem \ref{sound} we have
$$\max\limits_{\chi\in H} |L(\chi,1/2)|\ge \frac{\Re M_2}{M_1}\ge\frac{\#H}{\phi(q)}\max\limits_r\sum\limits_{nn_1\le N}\frac{r(n)r(nn_1)}{\sqrt{n}}\mathlarger{/}\sum\limits_{n\le N}|r(n)|^2+O(1)=$$
$$\exp\left(\left(\frac{1}{\sqrt{3}}+o(1)\right)\frac{\sqrt{\log q}}{\sqrt{\log\log q}}-\log[\mathcal{X}:H]\right).$$
\begin{remark*}
The constant $1/\sqrt{3}$ comes from the exponent in $N=q^{1/3}.$
\end{remark*}

\section{Non-trivial coset}
In this section we will prove Theorem \ref{thm2} in the introduction.

We consider maxima over a coset $\chi_1H$ of an arbitrary subgroup $H\le \mathcal{X}$, where $\chi_1\in \mathcal{X}$ is its representative.

By the same setup with
$$M_1=\sum_{\chi\in\chi_1H}|R^2(\chi)|=\#H\sum_{h\in\ker H}\chi_1(h)\sum_{n_1,n_2\le N}r(n_1)\overline{r(n_2)}\1{\congruence{hn_1}{n_2}{q}}.$$
$$M_2=\sum_{\chi\in\chi_1H}L(\chi,1/2)|R^2(\chi)|=O\left(NX^{1/2}q^{1/4}\sum_n |r(n)|^2\right)+\sum_{\chi\in\chi_1H}L^*(\chi,1/2)|R^2(\chi)|.$$
$$M_2+O\left(NX^{1/2}q^{1/4}\sum_n |r(n)|^2\right)=\sum_{n_1,n_2\le N}r(n_1)\overline{r(n_2)}\sum_{\chi\in \chi_1H}\chi(n_1)\overline{\chi(n_2)}L^*(\chi,1/2).$$

From here we set $X=1$, $N$ will be chosen small so it would not cause any issues, say $N\ll q^{1/5}$.

Now we have three cases: either $\chi_1H\subset\text{even}$, or $\chi_1H\subset\text{odd}$, or we observe that an intersection of cosets ($H$ and even or odd characters), is a coset of intersection of subgroups if non-empty.

Let us consider only the case when $H$ has both even and odd characters.

In other cases we do not split $\chi_1H$, since one has the same approximate functional equation for all $\chi\in\chi_1H$. So, we apply orthogonalities straight to $\chi_1H$. The only difference is the fact that the kernel remains itself, whereas for the splitting case the kernel will be $\ker(H\cap \mathcal{X}_0)=\ker H\cup-\ker H$, and we have two terms.

So the sum over characters due to the orthogonalities (Lemma \ref{principal_orthogonality} and Lemma \ref{dual_orthogonality}) is equal to the sum of the following expressions coming from principal and dual terms, denoted as $\text{PT}$ and $\text{DT}$ respectively

$$\text{PT}=\frac{\#H}{2}\sum_n \frac{\sum\limits_{h\in\ker H\cup-\ker H}\1{\congruence{h nn_1}{n_2}{q}}}{\sqrt{n}}(\chi'(h)V_0+\chi''(h)V_1)\left(\frac{n\sqrt{\pi}}{X\sqrt{q}}\right),$$

\begin{equation}\begin{split}\text{DT}=\frac{\#H}{2\sqrt{q}} \sum_n \frac{1}{\sqrt{n}}\Biggl(\sum\limits_{h\in\ker H\cup\{-\ker H\}}
\chi'\left(\frac{hn_1\overline{nn_2}}{q}\right)e\left(\frac{hn_1\overline{nn_2}}{q}\right)V_0\left(\frac{n\sqrt{\pi}X}{\sqrt{q}}\right)+\\\chi''\left(\frac{hn_1\overline{nn_2}}{q}\right)e\left(\frac{hn_1\overline{nn_2}}{q}\right)V_1\left(\frac{n\sqrt{\pi}X}{\sqrt{q}}\right)\Biggr),\end{split}
\end{equation}
where $\chi'$ and $\chi''$ are any representatives of the intersections of $\chi_1H$ with even and odd characters respectively as cosets of $H\cap\text{even}$.

Basically we say that
$\chi_1 H\cap \mathcal{X}_0=\chi'(H\cap \mathcal{X}_0),\;\;\; \chi_1 H\cap \mathcal{X}_1=\chi''(H\cap \mathcal{X}_0).$

Contribution of the dual sums is estimated as previously: we bound $\chi$-s and exponents by one, split the argument of $V_i$ into two ranges and the same estimate holds. See \eqref{errorterm}.

With the principal terms we cannot use the positivity trick since some values of $\chi'$ and $\chi''$ on the elements of $\ker H\cup-\ker H$ are negative. Hence we have to deal with $h\ne1$ in a different way. Also we have a similar issue with $M_1$: the expected value is $\#H\sum_n|r(n)|^2$, while bounding trivially we get $\phi(q)\sum|r(n)|^2$.

Now the main term is the contribution of $h=1$ with arguments of $V_i$ being less than 1. 
$$\text{MT}=\#H\sum_{nn_1\le N}\frac{r(n)\overline{r(nn_1)}}{\sqrt{n}}.$$

While the new error term comes from $h\ne1$ and argument of $V_i$ less than 1 (the case larger than 1 is bounded as in the full group case). Here we state a lemma for dealing with small $n$ in the error term.
\begin{lemma}\label{lemma5}
    Let $q$ be a large prime, $N\in\N$, $h\in G\le(\Z/q\Z)^\times$, $k=\#G>1$,
    and $n,n_1,n_2\in\N,\;\delta\in\R$ such that $n_1,n_2\le N=q^\nu,\; n\le q^\delta$ with $0<\nu<1$ and $0<\delta< \frac{1}{k-1}$.
    
    Then having $$\nu+\delta<\frac{1}{k-1}\left(1-\frac{\log k}{\log q}\right)$$ ensures the congruence $\congruence{hnn_1}{n_2}{q}$ has no solutions for $h\neq1$.
\end{lemma}
\begin{proof}
Assume $\congruence{hnn_1}{n_2}{q}$ and $h\neq1$.

One can raise the congruence to power $k$ to get $\congruence{(nn_1)^{k}}{n_2^{k}}{q}$ using $\congruence{h^{k}}{1}{q}$. Now we put both terms to the left-hand side and decompose into

$$\congruence{(nn_1-n_2)((nn_1)^{k-1}+\ldots+(n_2)^{k-1})}{0}{q}.$$
And since $h\ne1$ we get $\congruence{(nn_1)^{k-1}+\ldots+(n_2)^{k-1}}{0}{q}.$ If the left-hand side is less than $q$, then there are no solutions. It is enough that $$N^{k-1}\frac{n^{k}-1}{n-1}<q,$$
which is implied by the inequality on $\nu+\delta$.
\end{proof}

Now we apply the lemma with $G=\ker H\cup-\ker H$, $k=\#2\ker H$, so there is no contribution from $n\le q^\delta$.

Note that the assumptions of the lemma imply by a similar argument that in $M_1$ there will be no solutions at all for $h\ne1$. 

The error term becomes
$$\sum_{h\in\ker H\cup-\ker H\setminus\{1\}}\#H\sum_{n_1,n_2\le N}r(n_1)\overline{r(n_2)}\sum_{q^\delta\le n\le X\sqrt{q}/\sqrt{\pi}}\frac{\1{\congruence{hnn_1}{n_2}{q}}}{\sqrt{n}},$$
the sum over $n$ has at most one term due to congruence condition, and this term is less than $q^{-\delta/2}$, so it is less than
$$\ll\ker HNq^{-\delta/2}\cdot\#H\sum_n |r(n)|^2.$$

\begin{equation}
\text{So we want to maximize $N$ with the following constraints: }\begin{cases}\nu+\delta<\frac{1}{k-1}\left(1-\frac{\log k}{\log q}\right),\\ \#\ker H N= q^{\delta/2}.\end{cases} \phantom{\hspace{9cm}} 
\end{equation}
\begin{remark*}
We can afford the equality in the last condition as $1=O(1)$.
\end{remark*}

Now substitute $\#\ker H=q^\alpha$ and we get
$$\nu<\frac{1}{3(k-1)}\left(1-\alpha-\frac{\log 2}{\log q}\right)-\frac{2}{3}\alpha.$$

The final form is the same:
$$M_1=\#H\sum_{n\le N}|r(n)|^2,$$
$$M_2=\#H\sum_{nn_1\le N}\frac{r(n)\overline{r(nn_1)}}{\sqrt{n}}+O(M_1),$$
applying the theorem \ref{sound} we get the bound for $\max L(\chi,1/2)$.

$$\max_{\chi\in\chi_1H}|L(\chi,1/2)|\ge\exp{\left(\left(\sqrt{\nu}+o(1)\right)\mathcal{L}\right)}.$$
In the range $k=o\left(\frac{\log q}{\log\log q}\right)$ stated in the theorem extra terms in the expression for $\nu$ are small, so $$\nu=\frac{1}{\sqrt{k-1}}\left(\frac{1}{\sqrt{3}}+o(1)\right).$$

\section{Acknowledgements}

This paper has been written as a mémoire during M1 "Parcour Jacques Hadamard" at Paris-Saclay university. The author expresses gratitude to his advisor Asbjørn Christian Nordentoft for the pleasant working atmosphere, his suggestions and patience. Also he thanks Kevin Destagnol and Étienne Fouvry for their support and motivation.

The financial help of the Fondation Mathématique Jacques Hadamard is gratefully acknowledged.

\bibliographystyle{plain}

\end{document}